\documentclass[11pt]{article}
\usepackage[affil-it]{authblk}

\usepackage{amsmath,amssymb,amsthm,enumerate}
\usepackage{natbib}

\newtheorem{lemma}{Lemma}
\newtheorem{theorem}{Theorem}
\newtheorem{cor}{Corollary}

\begin{document}

\title{The Robot Crawler Model on Complete k-Partite and Erd\H{o}s-R{\'e}nyi Random Graphs}
\author{A. Davidson%
\thanks{email: \texttt{angus.davidson@bristol.ac.uk}} and A. Ganesh%
\thanks{email: \texttt{a.ganesh@bristol.ac.uk}}}
\affil{School of Mathematics, University of
Bristol, University Walk, Bristol BS8 1TW}

\setcitestyle{numbers}

\maketitle

\begin{abstract}
Web crawlers are used by internet search engines to gather information about the web graph. In this paper we investigate a simple process which models such software by walking around the vertices of a graph. Once initial random vertex weights have been assigned, the robot crawler traverses the graph deterministically following a greedy algorithm, always visiting the neighbour of least weight and then updating this weight to be the highest overall. We consider the maximum, minimum and average number of steps taken by the crawler to visit every vertex of firstly, complete k-partite graphs and secondly, sparse Erd\H{o}s-R{\'e}nyi random graphs. Our work follows on from a paper of Bonato et. al. who introduced the model.
\end{abstract}

\section{Introduction}



Using an analogy introduced by Messinger and Nowakowsk \cite{messinger2008robot}, heuristically the robot crawler model can be viewed as a robot cleaning the nodes of a graph according to a greedy algorithm. Upon arriving at a given vertex the robot ``cleans'' the vertex, and then moves to its ``dirtiest'' neighbour to continue the process. Crawlers are of practical use in gathering information used by internet search engines, (\cite{brin1998anatomy}, \cite{henzinger2004algorithmic}, \cite{olston2010web}). This particular version of the model was introduced by Bonato et. al. \cite{bonato2015robot} and we direct the reader to their paper for further insight into the problem's motivation and previous work done. There they considered the robot crawler performed on trees, complete k-partite graphs (with equal sized vertex classes), Erd\H{o}s-R{\'e}nyi random graphs and the preferential attachment model. The purpose of this paper is to offer an answer to open problems 1 and 2 posed there which relate to generalising their work concerning complete k-partite graphs and Erd\H{o}s-R{\'e}nyi random graphs.

The model introduced by Messinger and Nowakowsk \cite{messinger2008robot} is analogous to the robot crawler model, but the robot cleans edges, (which are weighted), rather than vertices. Models similar to those studied by Messinger and Nowakowsk \cite{messinger2008robot} were investigated by Berenbrink, Cooper and Friedetzky \cite{berenbrink2015random} and Orenshtein and Shinkar \cite{orenshtein2014greedy} who considered a class of random walks on graphs which prefer unused edges, although in their models the walker chooses independently among adjacent edges when they have all previously been traversed. 

Given a finite connected undirected simple graph $G = G(V,E)$ we fix from outset an initial weighting; a bijective function $w_0:V \to \{ -n, -n+1 ..., -1 \}$ indicating the initial ranking of how dirty the vertices are. Here and henceforth ``dirtiest''/``cleanest'' refers to the vertex with the lowest/highest weight in a given set. At time 1 the robot visits the ``dirtiest'' node in $V$, i.e. $w_0^{-1}(-n)$. At time $t \in \mathbb{N}$ the robot updates the weight of the vertex visited to $t$. So if the robot visits vertex $v$ at time $t$ then $w_t(v) = t$ and $w_t(v') = w_{t-1}(v') \text{ }\forall v' \in V, v' \ne v, t \in \mathbb{N}$. If all vertices then have positive weight, i.e. $\min_{y \in V}(w_t(y)) > 0$ then the algorithm terminates and we output $\mathcal{RC}(G,w_0) = t$; the number of steps taken to clean all vertices. Otherwise at time $t+1$ the robot moves to vertex $argmin \{ w_{t}(u): (u,v) \in E \}$ i.e. the dirtiest neighbour of $v$ at time $t$, and the process continues. As proved in \cite{bonato2015robot}, this algorithm will always terminate after a finite number of steps.

Using $\Omega_n$ to denote the set of ($n!$) initial weightings we define $\text{rc}(G) = \min_{w_0 \in \Omega_n}(\mathcal{RC}(G,w_0))$ and $\text{RC}(G) = \max_{w_0 \in \Omega_n}(\mathcal{RC}(G,w_0))$ the minimum and maximum number of steps needed to clean all vertices of G.

Now supposing $\overline{w_0}$ is a uniformly chosen element of $\Omega_n$ we define the average number of steps needed to clean all vertices of G; $\overline{\text{rc}}(G) = \mathbb{E}(\mathcal{RC}(G,\overline{w_0}))$.

\section{Complete k-Partite Graphs}

\subsection{Results}

Given some constants $c_1 \ge c_2 ... , \ge c_k$, $\sum_{i=1}^k c_i = 1$, $k \ge 3$ consider the robot crawler model performed on the complete $k$-partite graph $G_{n}$ induced by vertex sets $V_1, V_2, ..., V_k$ where $|V_i| = c_in$ $\forall 1 \le i \le k$.

\begin{theorem}
\hfill
\begin{enumerate}[(i)]
\item For $c_1 \le \frac{1}{2}$,  $\text{rc}(G_n) = n$
\item For $c_1 > \frac{1}{2}$,  $\text{rc}(G_n) = 2nc_1 -1$
\end{enumerate}
\end{theorem}

\begin{theorem}
\hfill
\begin{enumerate}[(i)]
\item For $c_2 \le \frac{1}{2}(1-c_1)$,  $\text{RC}(G_n) = n+c_{1}n-1$
\item For $c_2 > \frac{1}{2}(1-c_1)$,  $\text{RC}(G_n) = 2(n-c_{2}n)$
\end{enumerate}
\end{theorem}

\begin{theorem} 
\hfill
\begin{enumerate}[(i)]
\item For $c_1 < \frac{1}{2}$,  $\overline{\text{rc}}(G_n) = n+O(1)$
\item For $c_1 = \frac{1}{2}$,  $\overline{\text{rc}}(G_n) = n + O(n^{\frac{1}{2}})$
\item For $c_1 > \frac{1}{2}$,  $\overline{\text{rc}}(G_n) = 2nc_1 + O(1)$
\end{enumerate}
\end{theorem}

In particular we note that for $c_1 \neq \frac{1}{2}$, $\overline{\text{rc}}(G_n) = \text{rc}(G_n)+O(1)$, which refines Theorem 6 in \cite{bonato2015robot} if we take $G_n = K_{n/k}^{k}$, the complete $k$-partite graph induced by $k$ vertex sets each of size $\frac{n}{k}$.

\subsection{Proofs}

We begin with the more straightforward proofs of theorems 1 and 2.

\begin{proof}[Proof of Theorem 1]
It is straight forward to construct a Hamiltonian path to verify part (i). For part (ii) we note that once the crawler is in set $V_1$ (which takes at least 1 step) it must return at least $c_1n - 1$ times. Whenever the crawler is in set $V_1$ it will take at least 2 steps of the algorithm before the crawler returns since there are of course no edges between vertices in $V_1$. Hence $\text{rc}(G_n) \ge 1 +2(nc_1 -1)$. Noting that $|V_1| > |V \setminus V_1|$, the bound can be achieved if the crawler starts in $V_1$ and oscillates between $V_1$ and $V \setminus V_1$, e.g. if $w_0(v) < w_0(u)$ $\forall v \in V_1, u \in V \setminus V_1$.
\end{proof}

For $1 \le i \le k$ define the surplus of vertex set $i$ ($=:S_{w_0}(i)$) to be the number of uncleaned vertices remaining in $V_i$ at the moment all vertices in $V \setminus V_i$ have been cleaned. Clearly $S_{w_0}(i) = 0$ for all but one value of $i$. Further define $S_{w_0} = \sum_{i=1}^{k} S_{w_0}(i) = \max_{1 \le i \le k}(S_{w_0}(i))$. A crucial observation is that $\mathcal{RC}(G_n, w_0) = n + S_{w_0} - 1$. Indeed suppose $S_{w_0}(i) > 0$, then immediately after the time step ($t = n-S_{w_0}(i)$) when all vertices in $V \setminus V_i$ have been cleaned the crawler will alternate between $V_i$ and $V \setminus V_i$ until all remaining $S_{w_0}(i)$ uncleaned vertices of $V_i$ have been cleaned which will take a further $2S_{w_0}(i) - 1$ steps.

\begin{proof}[Proof of Theorem 2]
Clearly $S_{w_0} \le \max_{1 \le i \le k} |V_i| = c_1n$. Part (i) now amounts to showing that if  $c_2 \le \frac{1}{2}(1-c_1)$ then $\exists w_0$ such that $S_{w_0} = c_1n$. This follows in part since if $k \ge 4$ it is possible to clean $V \setminus V_1$ in $|V \setminus V_1|$ steps using Theorem 1 (i) on the complete $(k-1)$-partite graph induced by vertex sets $V_2, ..., V_k$, in which case $S_{w_0}(1) = c_1n$. Finally, if $k = 3$ then necessarily $c_2 = c_3$ and again it is of course possible to clean $V \setminus V_1$ in $|V \setminus V_1|$ steps simply by alternating between $V_2$ and $V_3$ for the first $2c_2 n$ steps.

Suppose now $c_2 > \frac{1}{2}(1-c_1)$ and $S_{w_0}(2) = 0$. When $V_2$ is fully cleaned there are uncleaned vertices elsewhere in $V$. We first note that it takes at least $2nc_2 - 1$ steps to clean all vertices of $V_2$ at which point there are at most $n-2nc_2+1$ vertices in $V$ not yet visited by the crawler. From this point it will take at most $2(n-2nc_2+1)-1$ steps to clean the remainder of the vertices, which gives the required upper bound  RC$(G_n) \le 2(n-2nc_2+1) -1 + 2nc_2 - 1 = 2n(1-c_2)$.  

Consider $w_0 \in \Omega_n$ with set $V_2$ being the $|V_2|$ dirtiest, and $V_1$ the $|V_1|$ cleanest vertices of $V$. That is $\bigcup_{j = 0}^{c_2n-1} w_0^{-1}(-n+j) = V_2$ and $\bigcup_{j = 1}^{c_1n} w_0^{-1}(-j) = V_1$, then the bound is attained.

\end{proof}

We now turn our attention to Theorem 3, the main result of the section. 

Let $m_{i} = \max(x: \exists y \ge 0$ s.t. $y+x$ of the $2y+x$ cleanest vertices lie in set $V_i)$. That is, $m_{i} = \max(x: \exists y \ge 0$ s.t. $\bigcup_{j = 1}^{2y+x}w_{0}^{-1}(-j) \cap V_i =y+x)$. Stochastically $m_{i}$ is the record of an $n$ step simple random walk, conditioned to be at a fixed position at time $n$. This random walk starts at the origin at time 0 and jumps up (down) by 1 at time $t$ if $w_{0}^{-1}(-t) \in V_i$ ($w_{0}^{-1}(-t) \in V \setminus V_i$), and finishes at time $n$ in position $|V_i| - |V \setminus V_i| = 2c_i - n$. More on this shortly.

\begin{lemma}
$S(i) \le m_i$.
\end{lemma}

\begin{proof}

W.l.o.g. take $i = 1$. Consider $x \in V_1$ defined to be the $(m_1+1)^{\text{st}}$ cleanest vertex in $V_1$, and suppose it is also the $(m_1+1+t)^{\text{th}}$ cleanest vertex in $V$ overall, (so $w_{0}(x) = -(m_1+1+t)$). So, there are $m_1$ vertices cleaner than $v$ in $V_1$ and $t$ cleaner than $v$ in $V \setminus V_1$. Clearly $t \ge 1$ by the definition of $m_1$. Let $v$ be the first vertex cleaned of the $(m_1+1+t)$ cleanest of $V$. If $v=x$ then we are done. If $v \ne x$ then $v \in V \setminus V_1$ and $v$ must have been cleaned immediately after some node $u \in V_1$ where $w_{0}(u) = -(m_1+1+t+l)$ some $l > 0$ and $\cup_{i=1}^{l} w_{0}^{-1}(-m_1-1-t-i) \subset V_1$. By the definition of $m_1$, necessarily $l < t$, (and all other nodes must have already been cleaned by the crawler). It is clear how the crawler will then proceed, alternating between $V_1$ and $V \setminus V_1$ until $x$ is cleaned at which point there will be $t-l>0$ uncleaned vertices in $V \setminus V_1$, and hence $S(1) \le m_1$.

\end{proof}

It is not difficult to construct a graph with some initial vertex weights such that $S(1) < m_1$. As a simple example, consider the complete 3-partite graph $G$  induced by $V_1, V_2, V_3$ with $|V_1| = 3, |V_2| = 3, |V_3| = 1$ and  $V_1$ consisting of the 3 cleanest vertices of $G$. In this case $m_1 = 3$ but $S(1) \le 2$.

We now make the link between $m_1$ and the record of a simple random walk bridge, (noting the start and end points of this bridge can be different). For $0 \le t \le n$ define $U(t) := |v \in V_1, w_0(v) \ge -t|$, the number of vertices in $V_1$ initially among the $t$ cleanest of $V$, $D(t) := |v \in V \setminus V_1, w_0(v) \ge -t| = t - U(t)$ and  $X(t) := U(t)-D(t)$. 

Let $(Z(t))_{t \ge 0}$ be a random walk on $\mathbb{Z}$ starting from $Z(0) = 0$ with $p = \mathbb{P}(Z(t+1)-Z(t)=1) = c_1$ and $q = \mathbb{P}(Z(t+1)-Z(t) = -1) = 1-p$ $\forall t \ge 0$. Observe that $(X(t))_{0 \le t \le n} \sim (Z(t)|Z(n)=|V_1|-|V \setminus V_1|)_{0 \le t \le n}$, and hence $X(t)$ is a random walk bridge starting at $X(0) = 0$ and ending at $X(n) = |V_1|-|V \setminus V_1|$. We could equally have defined $m_1 = \max_{0 \le t \le n} \{ X(t)  \}$. 

\begin{lemma}
For $c_i < 0.5$, $\mathbb{E}(m_i) \le  \frac{2c_i}{1-2c_i}$.
\end{lemma}

\begin{proof}
Again, w.l.o.g. take $i = 1$. Let $h_j = \mathbb{P}(\max_{t \ge 0}(Z(t)) \ge j)$. As a simple consequence of the Markov property, for $j \ge 1$:

$$
\begin{aligned}
h_j &= \mathbb{P}(Z(1) = 1)\mathbb{P}(\max_{t \ge 0}(Z(t)) \ge j|Z(1) = 1) \\
&\quad+\mathbb{P}(Z(1) = -1)\mathbb{P}(\max_{t \ge 0}(Z(t)) \ge j|Z(1) = -1) \\
 &= c_1 h_{j-1} +(1-c_1)h_{j+1}\\
\end{aligned}
$$

Using $c_1 < \frac{1}{2}$ together with the initial condition $h_0 = 1$ we find that $h_j = (\frac{c_1}{1-c_1})^{j}$ $\forall j \ge 0$. Now

$$
\begin{aligned}
\mathbb{P}(m_1 \ge j) &= \mathbb{P}(\max_{0 \le t \le n}(X(t)) \ge j) \\
&= \mathbb{P}(\max_{0 \le t \le n}(Z(t)) \ge j | Z(n)=|V_1|-|V \setminus V_1|) \\
&\le \mathbb{P}(\max_{0 \le t \le n}(Z(t)) \ge j | Z(n) \ge |V_1|-|V \setminus V_1|) \\
&\le \frac{\mathbb{P}(\max_{0 \le t \le n}(Z(t)) \ge j )}{\mathbb{P}(Z(n) \ge |V_1|-|V \setminus V_1|)} \\
&\le 2\mathbb{P}(\max_{0 \le t \le n}(Z(t)) \ge j) \\
\end{aligned}
$$

The first inequality follows from a simple coupling argument. If we are given a realisation of $(X(t))_{1 \le t \le n}$, and some integer $0 \le C \le |V_1|$, we can define the random path $(Z_1(t))_{1 \le t \le n}$ by taking $C$ of the down steps of $X(t)$ chosen uniformly at random among all of the ${|V_1|}\choose{C}$ possibilities and flipping them to up steps. Clearly, $Z_1(t) \ge X(t)$ $\forall 1 \le t \le n$, and hence $\max_{0 \le t \le n}(Z_{1}(t)) \ge \max_{0 \le t \le n}(X(t))$. If we initially let $C \sim \frac{1}{2}(Z(n) - (|V_1|-|V \setminus V_1|)) | Z(n) \ge |V_1| - |V \setminus V_1|$ then it is also clear that  $Z_1(t) \sim Z(t) | Z(n) \ge |V_1|-|V \setminus V_1|$.

Concluding the argument

$$
\mathbb{E}(m_1) \le 2\sum_{j = 1}^{\infty}\mathbb{P}(\max_{0 \le t \le n}(Z(t)) \ge j) = 2\sum_{j = 1}^{\infty}h_j = \frac{2c_1}{1-2c_1}
$$
\end{proof}

We can now conclude part (i) of Theorem 3. For $c_1 < 0.5$:

$$
\overline{\text{rc}}(G_n,w_0) = \mathbb{E}(n + S_{w_0}-1) \le n + \mathbb{E} \left( \sum_{i = 1}^{k} m_i \right) \le n + \sum_{i = 1}^{k} \frac{2c_i}{1-2c_i} = n+O(1)
$$

In proving Lemma 2, we linked $m_1$ to the maximum of a Random Walk Bridge $X(t)$ with the property that $X(0) > X(n)$, and eventually used the expected maximum level reached by a Random Walk with negative drift. Using a similar strategy to conclude part (ii) of Theorem 3 where $c_1 > 0.5$ wouldn't work since of course, the expected maximum reached by a Random Walk with positive drift is unbounded. To navigate this problem we will reverse time on the Random Walk Bridge.


\begin{cor}
For $c_i > \frac{1}{2}$, $\mathbb{E}(m_i) \le  2c_1 n -n + \frac{2(1-c_i)}{2c_i-1}$
\end{cor}

\begin{proof}

For $1 \le t \le n$ define $\widehat{X}(t) = X(n-t)$. $\widehat{X}(t)$ is again a Random Walk Bridge, but with $\widehat{X}(0) = 2c_1 n - n$ and $\widehat{X}(n) = 0$. The key point here is that $\left( \widehat{X}(t)|c_1 = \alpha \right) \sim \left( 2c_1n - n + X(t)|c_1 = 1-\alpha \right)$, so 
$$
\mathbb{E}(m_i) = \mathbb{E}\left( \max_{0 \le t \le n} \{ X(t)  \}\right) =  \mathbb{E}\left( \max_{0 \le t \le n} \{ \widehat{X}(t)  \}\right)  \le  2c_1 n -n + \frac{2(1-c_i)}{2c_i-1}
$$ 
by Lemma 2.

\end{proof}

We have now shown that for $c_1 > 0.5$, $\overline{\text{rc}}(G_n) \le n +\mathbb{E}(m_i) \le  2c_1 n + \frac{2(1-c_i)}{2c_i-1}$ which completes the proof of part (iii) of Theorem 3.

Finally, Godreche et. al. \cite{godreche2015record} prove that for $c_1 = 0.5$, $\mathbb{E}(\max_{0 \le t \le n}(X(t))) = \sqrt{\frac{\pi n}{8}}$. Part (ii) of Theorem 3 follows.




\section{Erdos-Renyi Random Graph}

We now turn our attention to open problem 2 in \cite{bonato2015robot}. In their paper Bonato et. al. considered the robot crawler performed on $G(n,p)$ with $np \ge \sqrt{n\log{n}}$. We will prove the 2 results in Theorem 4 below which are similar to Corollary 2 and Theorem 8 in their work, but for much sparser graphs:

\begin{theorem}
Let $p = f(n)\log{n}/n$ for some non-decreasing function $f > 28$. Then 
\begin{enumerate}[(i)]
\item ${\text{RC}}(G(n,p)) \le n^{2+o(1)}$ a.a.s. 
\item $\frac{ \mathcal{RC}(G(n,p),w_0)}{\left(n + \frac{n}{f(n)}\right)} \stackrel{p}{\longrightarrow} 1 \text{ as } n \to \infty$
\end{enumerate}
\end{theorem}


In particular we note that if $f(n) \to \infty$ as $n \to \infty$, however slowly, then $\frac{ \mathcal{RC}(G(n,p),w_0)}{n} \stackrel{p}{\longrightarrow} 1 \text{ as } n \to \infty$. 

\begin{proof}[Proof of Theorem 4 (i)]
We will use Lemma 1(5) from \cite{bonato2015robot} which states that for any graph $G$, ${\text{RC}}(G) \le n(\Delta + 1)^{d}$ where $\Delta$ is the maximum degree of a vertex in $G$, and $d$ is the diameter of $G$.

The number of neighbours of $v$, a typical vertex of $G(n,p)$, is distributed $Bin(n-1,p)$. Hence,

$$
\begin{aligned}
\mathbb{P}(v\text{ has } \ge 2np \text{ neighbours}) &= \mathbb{P}(Bin(n-1,p) \ge 2np) \\
&= (1+o(1))\mathbb{P}(\mathcal{N}((n-1)p,(n-1)p(1-p)) \ge 2np) \\
&\le (1+o(1))\Phi \left(\frac{-np}{\sqrt{(n-1)p(1-p)}} \right) \\
&\le (1+o(1))\Phi \left(-\sqrt{np} \right) \\
&\le (1+o(1)) \frac{e^{-np/2}}{\sqrt{2 \pi np}} \\
&\le (1+o(1)) n^{-f(n)/2} \le (1+o(1)) n^{-14}
\end{aligned}
$$

Hence by the union bound, $\mathbb{P}(\Delta \ge 2np) \le (1+o(1)) n^{-13}$

In a 2004 paper \cite{chung2001diameter}, (which extends the work of Bollobas \cite{bollobas1981diameter}), Chung and Lu showed that a.a.s., $d = (1+o(1))\frac{\log{n}}{\log(np)}$ for $np \to \infty$. Putting these bounds together, a.a.s;

$$
\begin{aligned}
n(\Delta + 1)^{d} &\le n(2np)^{(1+o(1))\frac{\log{n}}{\log{np}}} \\
&= n \exp((1+o(1))\frac{\log{n}}{\log{np}} \log{2np}) \\
&= n^{2+o(1)} = o(n^3)
\end{aligned}
$$

\end{proof}

To prove part (ii), we will have use for the following lemma:

%

\begin{lemma}	
Let $Y = \sum_{i = 1}^{n/7} X_i$ where $X_i \sim Geom(1-(1-p)^{i})$ independently for each $1 \le i \le n/7$. For all $\varepsilon > 0 $,
$$
\mathbb{P}\left((1-\varepsilon)\left(\frac{n}{7} +\frac{n}{f(n)}\right) < Y < (1+\varepsilon)\left(\frac{n}{7} +\frac{n}{f(n)}\right)\right) \stackrel{n \to \infty}{\longrightarrow} 1 
$$
\end{lemma}

\begin{proof}
To prove the upper bound we will use the following stochastic domination:

For $Z_1 \sim Geom(q)$ and $Z_2 \sim Exp(-\log(1-q))$, $Z_1 \preceq 1 + Z_2$.

Defining $E_i \sim Exp(i)$ for $1 \le i \le \frac{n}{7}$,

$$
Y \preceq \frac{1}{-\log(1-p)} \sum_{i = 1}^{n/7} E_i + \frac{n}{7}
$$

Hence,

$$
\begin{aligned}
\mathbb{P}\left(Y > (1+\varepsilon)\left(\frac{n}{7} +\frac{n}{f(n)}\right)\right) &\le \mathbb{P}\left(\frac{1}{-\log(1-p)} \sum_{i = 1}^{n/7} E_i > \frac{(1+\varepsilon)n}{f(n)} + \frac{\varepsilon n}{7}\right) \\
&\le \mathbb{P}\left(\frac{1}{-\log(1-p)} \sum_{i = 1}^{n/7} E_i > (1+\varepsilon) \frac{ n}{f(n)}\right) \\
&\le \mathbb{P}\left( \sum_{i = 1}^{n/7} E_i > (1+\varepsilon) \log{n}\right)
\end{aligned}
$$
since $-\log(1-p) \ge p = \frac{f(n)\log{n}}{n}$. Given that $\sum_{i = 1}^{n/7} E_i \sim \max_{1 \le i \le n/7} \{ E_1^{i} \}$ where $E_1^{i} \sim Exp(1)$ i.i.d., we apply the union bound to deduce

$$
\mathbb{P}\left(Y > (1+\varepsilon)\left(\frac{n}{7} +\frac{n}{f(n)}\right)\right) \le n e^{-(1 + \varepsilon)\log{n}} = n^{-\varepsilon} \stackrel{n \to \infty}{\longrightarrow} 0
$$

In proving the lower bound, we will use an even simpler stochastic domination:

For $T_1^{i} \sim Geom(1-(1-p)^{i})$ and $T_2^{i} \sim Geom(ip)$ with $i \ge 1, ip < 1$, $T_1^{i} \succeq T_2^{i}$. This follows from the simple inequality $1-(1-p)^{i} \le ip$ which holds $\forall i \ge 1$. Let $T =  \sum_{i = 1}^{n/f(n)\log{n}} T_2^{i}$. We find

$$
\begin{aligned}
\mathbb{P}\left(Y < (1-\varepsilon)\left(\frac{n}{7} +\frac{n}{f(n)}\right)\right) &\le \mathbb{P}\left(T + \sum_{i = 1+n/f(n)\log{n}}^{n/7}1 < (1-\varepsilon)\left(\frac{n}{7} +\frac{n}{f(n)}\right)\right) \\
&\le \mathbb{P}\left(T < \frac{n}{f(n)} -\frac{\varepsilon n}{7} + \frac{n}{f(n)\log{n}}\right)
\end{aligned}
$$

We recognise the relation between $T$ and the coupon collector problem, (see for example \cite{mitzenmacher2005probability}). Now,

$$
\begin{aligned}
\mathbb{E}\left(T\right) = \sum_{i = 1}^{n/f(n)\log{n}} \frac{1}{ip} &= \frac{\log{n} - \log{(f(n)\log{n})} + O(1)}{p} \\
&= \frac{n}{f(n)} - \frac{\log(f(n)\log{n}) - O(1)}{f(n)\log{n}} \\
\end{aligned}
$$

$$
\text{Var}\left(T \right) = \sum_{i = 1}^{n/f(n)\log{n}} \frac{1-ip}{(ip)^2} \le \frac{1}{p^2}\sum_{i = 1}^{\infty} \frac{1}{(i)^2} = \frac{\pi^{2}}{6p^2}
$$

We use Chebyshev's inequality to conclude

$$
\begin{aligned}
&\quad \text{    }  \mathbb{P}\left(Y < (1-\varepsilon)\left(\frac{n}{7} +\frac{n}{f(n)}\right)\right) \\
&\le \mathbb{P}\left(T < \frac{n}{f(n)} -\frac{\varepsilon n}{7} + \frac{n}{f(n)\log{n}}\right) \\
&\le \mathbb{P}\left( \left|T - \mathbb{E}\left( T \right)\right| > \mathbb{E}\left( T \right) - \frac{n}{f(n)} +\frac{\varepsilon n}{7} - \frac{n}{f(n)\log{n}}\right) \\
&\le \mathbb{P}\left( \left|T - \mathbb{E}\left( T \right)\right| > \frac{\varepsilon n}{7} - \frac{n(\log(f(n)\log{n}) + O(1))}{f(n)\log{n}}    \right) \\
&\le \mathbb{P}\left( \left|T - \mathbb{E}\left( T \right)\right| > \frac{\varepsilon n}{14} \right) \text{   (for large enough }n) \\
&\le \left(\frac{\varepsilon n}{14}\right)^{-2}\text{Var}\left(T \right) = \frac{196 \pi ^2}{6(\varepsilon f(n)\log{n})^2} \stackrel{n \to \infty}{\longrightarrow} 0
\end{aligned}
$$

\end{proof}



We will prove Theorem 4 (ii) by showing high probability lower/upper bounds on $\mathcal{RC}(G(n,p),w_0)$. This is achieved by showing that with high probability this robot crawler number dominates/is dominated by a particular sum of geometrics. We will then use Lemma 3 to reach the final conclusion.

Fix the order of the vertices of $G(n,p)$ by initial weighting before we realise the edges of the random graph. So w.l.o.g. $w_0(v_i) = -i$ $\forall 1 \le i \le n$.

\begin{proof}[Proof of Theorem 4 (ii)] \hfill

\textbf{Lower Bound}

We begin by showing $\mathbb{P}\left(\mathcal{RC}(G(n,p),w_0) \le (1-\varepsilon)\left(n+\frac{f(n)}{n}\right)\right) \stackrel{n \to \infty}{\longrightarrow} 0$,  the lower bound.

The crawler begins at time 1 at vertex $v_n$, initially the dirtiest node. 

Suppose that the crawler is positioned at vertex $v$, and that there are $i$ vertices yet to be visited. 

\begin{enumerate}[-]
\item If this is the crawler's first visit to $v$, no information is known about the presence of potential edges between $v$ and yet unvisited vertices, hence the probability that $v$ is connected to an unvisited vertex is $1 - (1-p)^{i}$ independently of all previous steps of the algorithm.
\end{enumerate}

Otherwise, suppose that $w$ was the vertex visited immediately after the crawler was last at vertex $v$.

\begin{enumerate}[-]
\item If $w$ had already been cleaned, then necessarily, it is cleaner than any yet unvisited vertex which implies there are no edges between $v$ and yet uncleaned vertices.
\item If $w$ had not already been cleaned, there are no edges between $v$ and any uncleaned vertices which are \textit{dirtier} than $w$, but presence of edges between $v$ and uncleaned vertices cleaner than $w$ is independent of all previous steps of the algorithm.
\end{enumerate}

In any case, the probability $v$ is connected to an unvisited vertex is $1 - (1-p)^{j}$ for some $0 \le j \le i$. Hence, independently of all previous steps of the process, the probability $v$ is connected to an unvisited vertex is less than $1 - (1-p)^{i}$.

This implies the number of steps needed before reaching the next yet uncleaned vertex dominates a $Geom(1 - (1-p)^{i})$ random variable, and $\forall \varepsilon > 0$
$$
\begin{aligned}
&\mathbb{P}\left(\mathcal{RC}(G(n,p),w_0) \le (1-\varepsilon)\left(n+\frac{f(n)}{n}\right)\right) \\
&\le \mathbb{P}\left(\sum_{i = 1}^{n-1} Geom(1-(1-p)^{i}) \le (1-\varepsilon)\left(n+\frac{f(n)}{n}\right)\right) \\
&\le \mathbb{P}\left(Y \le (1-\varepsilon)\left(\frac{n}{7}+\frac{f(n)}{n}\right)\right) \stackrel{n \to \infty}{\longrightarrow} 0 \\
\end{aligned}
$$
 by Lemma 3.

\textbf{Upper Bound}

It remains to show that $\forall \varepsilon > 0$

$$
\mathbb{P}\left(\mathcal{RC}(G(n,p),w_0) \ge (1+\varepsilon)\left(n+\frac{f(n)}{n}\right)\right) \stackrel{n \to \infty}{\longrightarrow} 0.
$$

As in \cite{bonato2015robot} we will consider different stages of the crawling process. 

\textit{Phase 1:} Again, the process will start from $v_n$, initially the dirtiest node and proceed to clean vertices of the graph. This phase ends when either of the following occur:
\begin{enumerate}[(a)]
\item $\frac{4n}{7}$ vertices have been cleaned.
\item The crawler is not adjacent to any of the $n/7$ dirtiest (and as yet uncleaned) vertices, which are necessarily contained in $ \{ v_i, i < \frac{5n}{7} \}$. 
\end{enumerate}

We define the jump number $J(v_i), (1 \le i \le n)$ of a vertex $v_i$ as the number of times any cleaner node was visited before it was first cleaned itself. Intuitively it is the number of potential edges connected to $v_i$ which were explored before one was first found, since each occurrence of a cleaner vertex being chosen by the crawler before vertex $v_i$ implies a missing edge between the crawlers position at that time and $v_i$.

If Phase 1 ends due to (a), and also the condition ``$J(v) \le n/7$ for all vertices'' at the end of Phase 1 holds we will say that property P1 holds.

During each step of the crawling process in Phase 1, potential edges between the crawler and dirty nodes are not yet exposed. At each step, event (b) occurs if $n/7$ unexplored edges are not present in $G(n,p)$. This occurs with probability at most

$$
(1-p)^{n/7} = (1-f(n)\log{n}/n)^{n/7} \le n^{-f(n)/7} = o(n^{-3})
$$

Hence by the union bound, with probability $1-o(n^{-2})$ Phase 1 ends due to (a). 

Further, as argued above ``$J(v_i) \ge n/7$'' implies that the first $n/7$ unexplored potential edges to $v_i$ were not present. Again this has probability at most $(1-p)^{n/7} = o(n^{-3})$, and hence by another application of the union bound, property P1 holds with probability $1-o(n^{-2})$.

An important point to note is that only edges between vertices in $ \{ v_i, i < \frac{5n}{7} \}$ have been explored. Crucially for Phase 3, property P1 implies that each vertex cleaned in this phase has had at most $2n/7$ potential edges exposed by the crawler.

\textit{Phase 2:} We continue to clean vertices until any one of the following occurs:
\begin{enumerate}[(a)]
\item The crawler is not adjacent to any as yet uncleaned vertex.
\item There are $n/7$ uncleaned vertices remaining in $G(n,p)$.
\end{enumerate}

If Phase 2 ends due to (b), and all vertices in $ \{ v_i, i < \frac{5n}{7} \}$ have been cleaned by the end of the phase, then we say property P2 holds.

As in Phase 1, Phase 2 ends due to (a) at each step if (at least) $n/7$ unexplored edges are not present in $G(n,p)$. Again we can conclude using the union bound that Phase 2 ends due to (b) with probability $1-o(n^{-2})$. 

Suppose now that $\exists v \in \{ v_i, i < \frac{5n}{7} \}$ such that $v$ has not been cleaned by the crawler by the end of Phase 2. This would imply that $J(v) \ge \frac{n}{7}$ which as previously calculated has probability $o(n^{-3})$.

Using this observation we again use the union bound to deduce:

$$
\begin{aligned}
&\mathbb{P}(\{ \text{P2 holds} \}|\{ \text{Phase 2 ends due to (b)} \} \cap \{ \text{P1 holds} \}) \\
= \text{ }&1-o(n^{-2})
\end{aligned}
$$

Hence, summarising what has been done so far, 
$$\mathbb{P}(\{ \text{P1 holds} \} \cap \{ \text{P2 holds} \}) = 1-o(n^{-2})$$

\textit{Phase 3:} During this phase the crawler will continue to visit yet uncleaned vertices of $G(n,p)$ as well as revisiting some of the vertices which were cleaned during Phase 1. These vertices will have the smallest weight at this stage. This phase ends when any of the following occur:

\begin{enumerate}[(a)]
\item The crawler is not adjacent to any yet uncleaned vertex nor to any vertex which was cleaned during Phase 1 and has not yet been revisited in Phase 3.
\item The phase takes longer than $2n/7$ steps.
\item All vertices are cleaned.
\end{enumerate}

If Phase 3 ends due to (c) then we say property P3 holds. In the explanation that follows, we condition on the event that P1 and P2 hold.

At each step of this phase, in total there are at least $3n/7$ ``target'' vertices which are yet to be visited at all or were cleaned in Phase 1 and have yet to be revisited in this phase. The reason for this is there are $4n/7$ vertices cleaned in Phase 1, $n/7$ vertices yet to be visited at all and this phase takes at most $2n/7$ steps. If the crawler has just revisited a vertex cleaned in Phase 1, P1 implies at most $2n/7$ potential edges adjacent to the vertex will have been explored earlier in the process, so at least $\frac{3n}{7} - \frac{2n}{7} = \frac{n}{7}$ potential edges to ``target'' vertices are still unexplored. Otherwise, if the crawler has just visited a vertex for the first time in the process then all ($\ge3n/7$) potential edges to ``target'' vertices are unexplored. This is because crucially: no edges between $ \{ v_i, i < \frac{5n}{7} \}$ and $ \{ v_i, i \ge \frac{5n}{7} \}$ are explored in Phase 1; P2 implies the uncleaned vertices at the beginning of Phase 3 are contained within $ \{ v_i, i \ge \frac{5n}{7} \}$ and as in earlier phases, the presence of potential edges between any possible current location of the crawler and yet unvisited vertices is still undetermined, and independent of previous steps of the process. Once again, the union bound tells us the probability we have (at least) $n/7$ unexplored edges not present in $G(n,p)$ during one of these steps, and hence that Phase 3 ends due to (a), is $o(n^{-2})$.

We now argue that with probability $o(n^{-2})$ Phase 3 ends due to (b). This is essentially a repeat of the argument in Phase 2. If Phase 3 ends due to (b) then $\ge \frac{n}{7}$ vertices cleaned in Phase 1 will have been revisited during Phase 3. If $v \in \{ v_i, i \ge \frac{5n}{7} \}$ is still uncleaned at the end of the phase, then $J(v) \ge \frac{n}{7}$, since all vertices cleaned in Phase 1 will be cleaner than $v$ before it is itself cleaned. Once again, this has probability $o(n^{-3})$ and applying the union bound:
$$
\begin{aligned}
&\mathbb{P}( \text{P3 holds} \} |\{ \text{Phase 3 ends due to (b) or (c)} \} \cap \{ \text{P2 holds} \} \cap \{ \text{P1 holds} \}) \\
= \text{ }&1-o(n^{-2})
\end{aligned}
$$
 
We can now conclude that:
$$
\mathbb{P}(\{ \text{P3 holds} \} | \{ \text{P1 holds} \} \cap \{ \text{P2 holds} \}) = 1-o(n^{-2})
$$

and hence bringing together earlier calculations 
$$
\mathbb{P}(\{ \text{P1, P2, P3 hold} \} ) = 1-o(n^{-2})
$$


If $\widehat{Y} := (Y| Y \le 2n/7)$ then conditional on P1, P2 and P3, Phases 1 and 2 will take $n - n/7$ steps and Phase 3 will take a number of steps distributed as $\widehat{Y}$. Indeed, during Phase 3 when there are $x$ yet uncleaned vertices in $ \{ v_i, i \ge \frac{5n}{7} \}$, (and hence $x$ unexplored edges from the crawlers current position and these vertices), the probability the crawler will be adjacent to at least one of them is given by $1-(1-p)^x$. If the crawler continues to visit vertices with unexplored edges to all $x$ yet uncleaned vertices then the probability the crawler will reach one of these $x$ vertices in the next $y$ steps is given by $\mathbb{P}(Geom(1-(1-p)^{x}) \le y)$. And so

$$
\begin{aligned}
&\mathbb{P}\left(\mathcal{RC}(G(n,p),w_0) \ge (1+\varepsilon)\left(n+\frac{f(n)}{n}\right)\right) \\
&\le  \mathbb{P} \left(\mathcal{RC}(G(n,p),w_0) \ge (1+\varepsilon)\left(n+\frac{f(n)}{n}\right) \Big| \{ \text{P1, P2, P3 hold} \}  \right) \\ 
&\quad + \mathbb{P}\left(\{ \text{P1, P2, P3 hold} \}^{C} \right) \\
&\le \mathbb{P} \left(\widehat{Y} + \frac{6n}{7} \ge (1+\varepsilon)\left(n+\frac{f(n)}{n}\right)  \right) + o(n^{-2}) \\
&\le \mathbb{P} \left(\widehat{Y} \ge (1+\varepsilon)\left(\frac{n}{7}+\frac{f(n)}{n}\right)  \right) + o(n^{-2})  \stackrel{n \to \infty}{\longrightarrow} 0
\end{aligned}
$$
again, by Lemma 3.


\end{proof}

\bibliography{Robot_Crawler}
\bibliographystyle{plain}

\end{document}